\documentclass[12pt]{article}

\usepackage{url}
\usepackage{amsthm}
\usepackage{amsmath}
\usepackage{fullpage}
\usepackage{enumerate}
\usepackage{amssymb}

\usepackage{booktabs}
\usepackage{float}

\newtheorem{thm}{Theorem}
\newtheorem{lem}{Lemma}

\newcommand{\D}{\rm d}

\title{Improved bounds on Brun's constant}
\author{Dave Platt\footnote{Supported by Australian Research Council Discovery Project DP160100932 and  EPSRC Grant EP/K034383/1.}\\ School of Mathematics \\ University of Bristol, Bristol, UK\\dave.platt@bris.ac.uk\and  Tim Trudgian\footnote{Supported by Australian Research Council Discovery Project DP160100932 and Future Fellowship FT160100094.} \\
School of Physical, Environmental and Mathematical Sciences\\ The University of New South Wales Canberra, Australia \\
  t.trudgian@adfa.edu.au}
\begin{document}
\maketitle
\centerline{
Dedicated to the memory of Jon Borwein.}
\begin{abstract}
\noindent
Brun's constant is $B=\sum_{p \in P_{2}} p^{-1} + (p+2)^{-1}$, where the summation is over all twin primes. We improve the unconditional bounds on Brun's constant to $1.840503 < B < 2.288513$, which is about a 13\% improvement on the previous best published result.
\end{abstract}

\section{Introduction}
Brun \cite{Brun} showed that the sum of the reciprocals of the twin primes converges. That is, if $P_{2}$ denotes the set of primes $p$ such that $p+2$ is also prime, the sum $B:= \sum_{p\in P_{2}} 1/p + 1/(p+2)$ is finite. 

Various estimates for Brun's constant have been given based on calculations of $\pi_{2}(x)$, where $\pi_{2}(x)$ denote the number of twin primes not exceeding $x$ --- see Brent \cite[pp.\ 50--53]{Brent} and Klyve \cite[Table 1.2.3]{Klyve} for some historical references. Brent \cite{Brent} computed $\pi_{2}(8\cdot 10^{10}) = 182\, 855\, 913$, and, conditional on some assumptions about the random distribution of twin primes, conjectured that
\begin{equation}\label{sock}
B = 1.9021604\pm 5 \cdot 10^{-7}.
\end{equation}
Additional computations were undertaken by Gourdon and Sebah \cite{GS} and Nicely\footnote{We cannot resist referencing an anecdote from Jon Borwein himself (and his co-authors). In \cite[p.\ 40]{OM} Borwein, Borwein, and Bailey talk of Nicely's calculations on Brun's constant. Nicely discovered a bug in an Intel Pentium chip, which, according to \cite{OM} `cost Intel about a billions dollars'. We believe Jon would have seen this as an excellent application of pure mathematics in the modern world.} \cite{Nicely}, who showed 
\begin{equation}\label{cough}
\pi_{2}(2\cdot 10^{16}) =  19\,831\,847\,025\,792.
\end{equation}
Additionally, Nicely conjectured that
\begin{equation}\label{shoe}
B = 1.90216 05832 09 \pm 0.00000 00007 81.
\end{equation}

As far as we are aware the most comprehensive results on the enumeration of $\pi_{2}(x)$ are by Tom\'as Oliveira e Silva \cite{Silva}, who computed $\pi_2(k\cdot 10^{n})$ for $k=1,\ldots, 10\,000$ and $n=1,\ldots, 14$ and $\pi_2(k\cdot 10^{15})$ for $k=1,\ldots, 4\,000$.

Some explanation is required for these conjectured bounds in (\ref{sock}) and (\ref{shoe}). These results are not strict error bounds, but rather, confidence intervals (in the probabilistic sense). One can obtain a lower bound on $B$ by merely summing $B(N):= \sum_{p\in P_{2}, p \leq N} 1/p + 1/(p+2)$ for large values of $N$. One can then plot this as a function of $N$, make assumptions about the random distribution of twin primes, and try to ascertain the rate of convergence. This is what has been done by Brent, Nicely, and others.

It is another matter to ask for a rigorous upper bound for Brun's constant; clearly computing the sum $B(N)$ for any $N$ gives a lower bound. The first upper bound appears to be given by Crandall and Pomerance \cite{Crandall}, who showed that $B< 2.347$. An excellent exposition of their proof is given in a thesis by Klyve \cite{Klyve} who also shows that under the assumption of the Generalised Riemann Hypothesis we have $B< 2.1754$.

It is perhaps curious that the method of Crandall and Pomerance produces an upper bound for $B$ that is dependent on the lower bound.  When one increases the value of $N$, the corresponding increase in $B(N)$ yields a better upper bound for $B$.

In this paper we do two things: we compute $B(N)$ for a larger $N$ than was done previously, and using some optimisation improve the upper bound for $B$. The result is

\begin{thm}\label{jack}
$1.840503 < B < 2.288513$.
\end{thm}

The previous best lower bound was computed by Nicely \cite{Nicely}, who, using his calculations of (\ref{cough})  showed that $B(2\cdot10^{16})> 1.831808$. We remark that the lower bound of $B(10^{16})>1.83049$ by Gourdon and Sebah \cite{GS} was used in the proof of Crandall and Pomerance. 

In \S \ref{compB} we give details of using the tables by Oliveira e Silva in \cite{Silva} to compute $B(4\cdot 10^{18})$. This proves the lower bound in Theorem \ref{jack}. We remark here that this computation on its own would give an upper bound of $2.292$ in Theorem \ref{jack}.

In \S \ref{Prep} we list two results in the literature, one an explicit bound on a sum of divisors, and another an improvement on a sieving inequality used by Montgomery and Vaughan \cite{MV}. In \S \ref{RVsec} we introduce Riesel and Vaughan's bounds for $\pi_{2}(x)$. Finally, in \S \ref{chalk} we perform our calculations that prove the upper bound in Theorem \ref{jack}, and outline some of the difficulties facing future investigations into this problem.

\section{Preparatory results}\label{Prep}
We require two results from the literature. The first is an explicit estimate on $\sum_{n\leq x} d(n)/n$, where $d(n)$ is the number of divisors of $n$; the second is a large-sieve inequality.

\subsection{Bounds on sums of divisors}

The classical bound on   $\sum_{n\leq x} d(n)$ and partial summation show that
\begin{equation}\label{hearts}\sum_{n\leq x} \frac{d(n)}{n} \sim \frac{1}{2} \log^{2} x.
\end{equation}
It is also possible to give an asymptotic expansion of the above relation. First, for $k$ a non-negative integer, define the Stieltjes constants $\gamma_{k}$ as
\begin{equation*}
\gamma_{k} = \lim_{N\rightarrow\infty} \left\{- \frac{(\log N)^{k+1}}{k+1} + \sum_{n\leq N} \frac{(\log n)^{k}}{n} \right\}.
\end{equation*}
Here $\gamma_{0} = \gamma$, which is Euler's constant. In what follows we only need the following bounds: more precision is possible, but the estimates in (\ref{dodo}) are more than sufficient.
\begin{equation}\label{dodo}
0.5772156< \gamma_{0} < 0.5772157, \quad -0.0728159 < \gamma_{1} < -0.0728158.
\end{equation}
 Riesel and Vaughan give a more refined estimation of (\ref{hearts}), namely, if
\begin{equation}\label{queen}
E(x) = \sum_{n\leq x} \frac{d(n)}{n} - \frac{1}{2} \log^{2} x - 2 \gamma_{0} \log x - \gamma_{0}^{2} + 2 \gamma_{1},
\end{equation}
then by Lemma 1 \cite{RV}
\begin{equation}\label{king}
|E(x)| < 1.641 x^{-1/3}, \quad (x>0).
\end{equation}
We note that an improvement to this is claimed in Corollary 2.2 in \cite{Berk} which gives
\begin{equation*}
|E(x)| < 1.16 x^{-1/3}, \quad (x>0).
\end{equation*}
This, however, appears to be in error, since, as shown in \cite[p.\ 50]{RV} the error $|E(x)|x^{1/3}$ has a maximum of $-1.6408\ldots$ around $7.345\cdot 10^{-4}$. It is possible to improve (\ref{king}) by choosing an exponent smaller than $-1/3$. While this has only a minor impact on the estimation of Brun's constant, we record it below as it may be of interest elsewhere.
\begin{lem}\label{tea}
Let $E(x)$ be as in (\ref{queen}). Then, for all $x>0$ we have $|E(x)| \leq 1.0503 x^{-2/5}$.
\end{lem}
\begin{proof}
We proceed as in the proof of Lemma 1 in \cite{RV}. There, the authors consider three ranges, $x\geq 2$, $1\leq x <2$ and $0< x <1$. The idea with such a proof is by considering sufficiently many ranges, one can show that the global maximum of $|E(x)|x^{\alpha}$ occurs in $0<x<1$. By reducing  $\alpha$ we reduce this maximum value.  We find that writing $(1, \infty)$ as the union of $[n, n+1)$ for $1\leq n \leq 7$ and $[8, \infty]$ keeps the other contributions sufficiently small and establishes the lemma.
\end{proof}
We remark that the proof is easily adaptable to finding, for a given $\alpha$, the optimal constant $c= c(\alpha)$ such that $|E(x)|x^{\alpha} \leq c$ for all $x>0$. However, as we show in \S \ref{jute}, the effects of further improvements are minimal.

\subsection{A large sieve inequality}\label{band}

Riesel and Vaughan make use of the following, which is Corollary 1 in \cite{MV}.

\begin{thm}[Montgomery and Vaughan]\label{MVT1}
Let $\mathcal{N}$ be a set of $Z$ integers contained in $[M+1, M+N]$. Let $\omega(p)$ denote the number of residue classes mod $p$ that contain no element of $N$. Then $Z\leq L^{-1}$, where
\begin{equation}\label{MVT}
L = \sum_{q\leq z} \left( N+ \frac{3}{2} qz\right)^{-1} \mu^{2}(q) \prod_{p|q} \frac{\omega(p)}{p - \omega(p)},
\end{equation}
where $z$ is any positive number.
\end{thm}
Actually, Theorem \ref{MVT1} is derived from the investigations of Montgomery and Vaughan into Hilbert's inequality. Specifically, Theorem \ref{MVT1} follows from Theorem 1 in \cite{MV2}. That result was improved by Preissmann \cite{Prize}. The upshot of all this is that Preissmann's work allows one to take $\rho = \sqrt{1 + 2/3 \sqrt{6/5}} \approx 1.315\ldots$ in place\footnote{\label{links}We remark that Selberg conjectured that (\ref{MVT}) holds with $1$ in place of $3/2$. It seems difficult to improve further on Preissmann's work.}  of $3/2$ in (\ref{MVT}). 

Riesel and Vaughan choose $z = (2x/3)^{1/2}$ in (\ref{MVT}). With Preissman's improvement we set $z = (x/\rho)^{1/2}$; it is trivial to trace through the concomitant improvements.

\section{Riesel and Vaughan's bounds on $\pi_{2}(x)$}\label{RVsec}
Riesel and Vaughan give a method to bound $\pi_{2}(x)$. Actually, their method is much more general and can bound the number of primes $p\leq x$ such that $ap+b$ is also prime. We present below their method for the case of interest to us, namely, that of $a=1, b=2$.

We first let $C$ denote the twin prime constant 
\begin{equation}\label{milk}
C = 2 \prod_{p>2} \frac{p(p-2)}{(p-1)^{2}}.
\end{equation}
Note that in some sources the leading factor of $2$ may be absent. Wrench \cite{Wrench} computed $C$ to 45 decimal places. For our purposes the bound given by Riesel and Vaughan below is sufficient
\begin{equation*}
1.320323<C<1.320324.
\end{equation*}

\begin{lem}\label{lem:pi2}
For any $s>-1/2$ we define $H(s)$ by
\begin{equation*}
H(s) = \sum_{n=1}^{\infty} \frac{|g(n)|}{n^{s}},
\end{equation*}
where $g(n)$ is a multiplicative function defined by
\begin{equation*}
\begin{split}
g(p^{k}) &= 0 \;\textrm{for}\; k>3, \quad g(2) = 0, \quad g(4) = -3/4, \quad g(8) = 1/4,\\
g(p) &= \frac{4}{p(p-2)}, \quad g(p^{2}) = \frac{-3p -2}{p^{2}(p-2)}, \quad g(p^{3}) = \frac{2}{p^{2}(p-2)}, \quad (\textrm{when}\; p>2).
\end{split}
\end{equation*}

Now define the constants $A_i$ by
\begin{equation*}
\begin{split}
A_{6} &= 9.27436 - 2 \log \rho \\ 
A_{7} &= -5.6646 + \log^{2} \rho  - 9.2744\log \rho\\
A_{8} &= 16Cc(\alpha)H(-\alpha) \rho^{\alpha/2}\\
A_{9} &= 24.09391\rho^{1/2},
\end{split}
\end{equation*}
where $c(\alpha)$ is such that $|E(x)| x^{\alpha} \leq c(\alpha)$ for all $x>0$.

Now let
\begin{equation}\label{coffee}
F(x) = \max\left\{0, A_{6} + \frac{A_{7}}{\log x} - \frac{A_{8}}{x^{\alpha/2} \log x} - \frac{A_{9}}{x^{1/2} \log x}\right\}.
\end{equation}

Then
\begin{equation}\label{strainer}
\pi_{2}(x) < \frac{8Cx}{(\log x)(\log x + F(x))} + 2x^{1/2}.
\end{equation}
\begin{proof}
See \cite{RV}, equation ($3.20$).
\end{proof}
\end{lem}


This leads directly to the following lemma.

\begin{lem}\label{lem:B}
Let $F(x)$ be as defined in (\ref{coffee}). Chose $x_0$ large enough so that $F(x_0)>0$ and set
$$
B(x_0)=\sum\limits_{\substack{p\in P_2\\p\leq x_0}} \frac{1}{p}+\frac{1}{p+2}.
$$
Then
$$
B\leq B(x_0)-2\frac{\pi_2(x_0)}{x_0}+\int\limits_{x_0}^\infty\frac{16C}{t\log(t)(\log(t)+F(t))}+4t^{-\frac32}\D t.
$$
\begin{proof}
We start from
$$
B\leq B(x_0)+\sum\limits_{\substack{p\in P_2\\p>x_0}}\frac{2}{p}=B(x_0)+2\int\limits_{x_0}^\infty \frac{\D\pi_2(t)}{t},
$$
integrate by parts and apply Lemma \ref{lem:pi2}.
\end{proof} 
\end{lem}
Riesel and Vaughan calculate $H(-1/3)$ so that they may use (\ref{king}); we proceed to give an upper bound for $H(-2/5)$ in order to use Lemma \ref{tea}.

\begin{lem}\label{lem:H}
Let $H$ be as defined above, then
$$
H\left(-\frac25\right)<950.05.
$$
\begin{proof}
Write
$$
g(2,s)=\log\left(1+\frac{3}{4}2^{-2s}+\frac{1}{4}2^{-3s}\right)
$$
and for $t>2$
$$
g(t,s)=\log\left(1+\frac{4}{t(t-2)}t^{-s}+\frac{3t+2}{t^2(t-2)}t^{-2s}+\frac{2}{t^2(t-2)}t^{-3s}\right)
$$
so that for $s>-1/2$ we have the Euler product
$$
H(s)=\exp\left[\sum\limits_p g(p,s)\right].
$$
Now fix $P>2$ and split the sum into
$$
S_1(P,s)=\sum\limits_{p\leq P} g(p,s)
$$
and
$$
S_2(P,s)=\sum\limits_{p>P}g(p,s).
$$
Then by direct computation using interval arithmetic we find
$$
S_1\left(10^{10},-\frac{2}{5}\right)=6.8509190277\ldots
$$
To estimate $S_2$ we write
$$
\sum\limits_{p>P}g(p,s)=\int\limits_P^\infty g(t,s)\D \pi(t)\leq \int\limits_P^\infty \log\left(1+k_1 t^{-\frac65}\right)\D \pi(t)
$$
where $k_1$ is chosen so that $\log\left(1+k_1 t^{-\frac65}\right)\leq g\left(P,-\frac{2}{5}\right)$. For $P=10^{10}$ we find that $k_1=3.000402$ will suffice. We then integrate by parts to get
$$
S_2\left(P,-\frac25\right)\leq -\log\left(1+k_1 t^{-\frac65}\right)\pi(P)+\int\limits_P^\infty\frac65\frac{k_1}{t^{11/5}+k_1t}\pi(t)\D t.
$$

We compute the first term using $\pi\left(10^{10}\right)=455\,052\,511$ and for the second term we note that for $x\geq P$ we have
$$
\pi(x)\leq\frac{x}{\log x}\left(1+\frac{1.2762}{\log P}\right)=k_2\frac{x}{\log x}.
$$
The integral is now
$$
\frac{6}{5}k_1k_2\int\limits_P^\infty\frac{\D t}{\log t\left(t^{6/5}+k_1\right)}\leq\frac{6}{5}k_1k_2\int\limits_P^\infty\frac{\D t}{t^{6/5}\log t}=-\frac{6}{5}k_1k_2\textrm{Ei}\left(-\frac{\log P}{5}\right)
$$
where $\textrm{Ei}$ is the exponential integral
$$
\textrm{Ei}(x)=-\int\limits_{-x}^\infty\frac{\exp(-t)}{t}\D t.
$$
Putting this all together we have
$$
S_1\left(10^{10},-\frac{2}{5}\right)+S_2\left(10^{10},-\frac{2}{5}\right)< 6.8509191-0.0013653+0.0069531=6.8565069
$$
and thus $H\left(-\frac25\right)<950.05$.
\end{proof}
\end{lem}

\section{Calculations}\label{chalk}
We now have everything we require to prove Theorem \ref{jack}. We first proceed to the lower bound.
\subsection{Computing $B(4\cdot 10^{18})$: the lower bound in Theorem \ref{jack}}\label{compB}

We first note the following.
\begin{lem}\label{lem:pi2x0}
We have
$$
\pi_2\left(4\cdot 10^{18}\right)=3\,023\,463\,123\,235\,320.
$$
\begin{proof}
See \cite{Silva}, table ``2d15.txt''.
\end{proof}
\end{lem}

Furthermore, typical entries in the tables in \cite{Silva} (``2d12.txt'' for this example) look like

$$
1000d12\;\;\;\;  1177209242304\;\;\;\;  1177208491858.251\ldots
$$
$$
1001d12\;\;\;\;  1178316017996\;\;\;\;  1178315253072.811\ldots
$$
where the second column gives the count of prime pairs below the value given in the first column, interpreting, for example, ``1001d12'' as $1001\cdot 10^{12}$. From this we conclude that there are $1\,178\,316\,017\,996-1\,177\,209\,242\,304=1\,106\,775\,692$ prime pairs between $1000\cdot 10^{12}$ and $1001\cdot 10^{12}$. The contribution these will make to the constant $B$ is at least
$$
1\,106\,775\,692\times\frac{2}{1001\cdot 10^{12}}>1.0567\cdot 10^{-6}
$$
and at most
$$
1\,106\,775\,692\times\frac{2}{1000\cdot 10^{12}}<1.0678\cdot 10^{-6}.
$$

We take the value of $B(10^{12})\in[1.8065924,1.8065925]$ from \cite{Nicely} and add on the contributions from the entries in the tables from \cite{Silva} to conclude the following.
\begin{lem}\label{lem:B_bound}
$$
B\left(4\cdot 10^{18}\right)\in[1.840503,1.840518].
$$
\end{lem}
We note that the lower bound in Theorem \ref{jack} follows from Lemma \ref{lem:B_bound}. We note further that we are `off' by at most $1.5\cdot 10^{-5}$, which shows that there is limited applicability for a finer search of values of $\pi_{2}(x)$ for $x\leq 4\cdot 10^{18}$.

\subsection{The upper bound in Theorem \ref{jack}}
We shall use Lemma \ref{lem:B} to bound $B$.
Using $s=2/5$  to get $H(-2/5) < 950.05$ (Lemma \ref{lem:H}) and  $c(2/5)<1.0503$ (Lemma \ref{tea}) we get

\begin{equation*}
A_{6} > 8.72606, \quad A_{7} > -8.13199, \quad 
A_{8} < 22267.54, \quad 
A_{9} < 27.63359.
\end{equation*}

We chose $x_0=4\cdot 10^{18}$ so that $\pi_2(x_0)=3\,023\,463\,123\,235\,320$ (Lemma \ref{lem:pi2x0}) and $B(x_0)<1.840518$ (Lemma \ref{lem:B_bound}). This leaves the evaluation of
$$
\int\limits_{x_0}^\infty \frac{\D t}{t\log t(F(t)+\log t)}.
$$

We proceed using rigorous quadrature via the techniques of Molin \cite{Molin} implemented using the ARB package \cite{ARB} to compute
$$
\int\limits_{x_0}^{\exp(20\,000)} \frac{\D t}{t\log t(F(t)+\log t)}
$$

and then we bound the remainder by

$$
\int\limits_{\exp(20\,000)}^\infty \frac{\D t}{t\log t(F(t)+\log t)}\leq\int\limits_{\exp(20\,000)}^\infty \frac{\D t}{t\log^2 t}=\frac{1}{20\,000}.
$$

This establishes Theorem \ref{jack}.

\subsection{Potential Improvements}\label{jute}

We close this section by considering potential improvements whilst still relying on Riesel and Vaughan's method. One approach is to attempt to improve the constants $A_i$. A second would be to compute $B(x_0)$ for larger values of $x_0$ than the $4\cdot 10^{18}$ used above.

\subsubsection{Improving the constants $A_i$}

In the following, all calculations were done with $x_{0} = 4\cdot 10^{18}$, cutting off at $\exp(20\, 000)$, and using Preissmann's value for $\rho$ in \S \ref{band}.
\begin{enumerate}
\item The `$2$' that appears in (\ref{strainer}) is a result of the term $2\pi(z) + 1$ appearing on \cite[p.\ 54]{RV}. With the choice of $z = (x/\rho)^{1/2}$, and using the bound $\pi(x) < 1.25506 x/\log x$ from Rosser and Schoenfeld \cite[(3.6)]{RS}, we could replace the $2$ by 
\begin{equation*}
 x_{0}^{-1/2} + \frac{5.03}{\rho^{1/2} \log \frac{x_{0}}{\rho}} =0.10305\ldots
 \end{equation*}
\item We can replace the constant $A_9$ by $19.638\rho^{1/2} < 22.523$ by a careful examination of the final part of the proof of Lemma 3 in \cite{RV}.
\item We could investigate other versions of Lemma \ref{tea}. This would have the affect of reducing $A_{8}$. It should be noted that for larger values of $\alpha$ one can obtain smaller constants $c(\alpha)$ at the expense of a larger, and more slowly converging, $H(-\alpha)$. We did not pursue the optimal value of $\alpha$.
\end{enumerate}
However, we observe that setting $A_6=9.27436$ (that is, assuming Selberg's conjecture, in the footnote on page \pageref{links}, that $\rho =1$) and setting $A_7=A_8=A_9=0$ and deleting the $x^{1/2}$ term from (\ref{strainer}) altogether only reduces the upper bound for $B$ to $2.28545\ldots$. 


\subsubsection{Increasing $x_0$}

Knowledge of $B(x_0)$ and $\pi_2(x_0)$ for larger $x_0$ would allow us to further improve on our bound for $B$. To quantify such improvements, recall that results such as (\ref{sock}) and (\ref{shoe}) are obtained by assuming the Hardy--Littlewood conjecture, namely
\begin{equation}\label{syrup}
\pi_{2}(x) \sim C \int_{2}^{x} \frac{dx}{\log^{2} x},
\end{equation}
(where $C$ is the twin prime constant in (\ref{milk})), 
and assuming properties on the distribution of twin primes. This leads to the hypothesis that
\begin{equation}\label{lace}
B(n) \approx B - \frac{2C}{\log n}.
\end{equation}
Using (\ref{syrup}) and (\ref{lace}), one can `predict' the value of $\pi_{2}(10^{k})$ and $B(10^{k})$ for higher values of $k$. Of course one can object at this point: we are assuming a value of $B$ in order to obtain an upper bound on $B$! A valid point, to be sure. The purpose of this commentary is instead to show that without new ideas, this current method is unlikely to yield `decent' bounds on $B$ even using infeasible computational resources.

We ran the analysis from \S \ref{chalk} (not optimised for each $k$) to obtain the following.

\begin{table}[ht]
 \centering
 \caption{Projected upper bounds on $B$}
 \medskip
 \begin{tabular}{c c c c}
   \hline

   $k$& $B(10^{k})$  & $\pi_{2}(10^{k})$ &  Upper bound for $B$ \\
   \hline
      19 & 1.84181 & $7.2376\cdot 10^{15}$ & 2.2813\\
        20 & 1.84482 & $6.5155\cdot 10^{16}$ & 2.2641\\
          80 & 1.8878 & $3.9341\cdot 10^{75}$ &1.9998\\
   \hline

   \hline
 \end{tabular}

 \end{table}
Therefore, proving even that $B<2$ is a good candidate for the 13th Labour of Hercules, a man referenced frequently in puzzles by the late Jon Borwein.

\end{document}